\documentclass[a4paper,10pt]{article}
\pdfoutput=1
\usepackage{amsmath,amsfonts,amssymb,amsthm}
\usepackage{fullpage}
\usepackage{url}
\usepackage{graphicx, epstopdf}
\newtheorem{thm}{Theorem}[section]

\newtheorem{prop}[thm]{Proposition}
\newtheorem{ques}[thm]{Question}
\newtheorem{fact}[thm]{Fact}
\theoremstyle{definition}
\newtheorem{defn}[thm]{Definition}
\theoremstyle{remark}

\newtheorem*{rem}{Remark}
\newtheorem{khprop}[thm]{Property}
\title{Rotors in Khovanov Homology}
\author{Joseph MacColl}
\date{}

\begin{document}

\maketitle

\begin{abstract}
Anstee, Przyticki and Rolfsen introduced the idea of rotants, pairs of links related by a generalised form of link mutation \cite{A-P-R}. We exhibit infinitely many pairs of rotants which can be distinguished by Khovanov homology, but not by the Jones polynomial.
\end{abstract}

\section{Introduction}\label{sec:bground}

The following is a long-standing open problem in low-dimensional topology:

\begin{ques}\label{ques:jones}
Does there exist a nontrivial knot with trivial Jones polynomial?
\end{ques}

This question has led to operations on links which can alter them while leaving their Jones polynomial fixed. It is known that mutation as originally defined by Conway is such an operation, but that this mutation cannot take an unknot to a nontrivial knot \cite{Rolfsen}. We consider a generalised mutation operation introduced by Anstee, Przyticki and Rolfsen \cite{A-P-R}:

\begin{defn}\label{defn:rotant}
Let $n\geq 3$, and $L$ be a fixed planar projection of a link containing a $2n$-ended tangle $R$. If $R$ possesses $n$-fold rotational symmetry about its centre then $R$ is called a \emph{rotor}, or \emph{$n$-rotor}, and $S:=L\setminus R$ is called the \emph{stator}. We rotate the rotor $R$ by $\pi$, through the third dimension, and reinsert it into $L$ to obtain a new link $L^R$. If we can fix orientations on $L$ and $L^R$ so that their writhes agree, then we say that the links are \emph{(order $n$) rotants}, or \emph{$n$-rotants}, of each other.
\end{defn}
See Figure~\ref{fig:first ex} for an example of order 4 rotants.
\\
\par
In the case of rotants, we know that the generalised mutation can preserve the Jones polynomial; indeed, the following was proven in \cite{A-P-R}, where the idea of rotants was introduced:

\begin{thm}\label{thm:rotant jones} 
Suppose $n\leq 5$. If the link $L^R$ is an $n$-rotant of $L$, then $L$ and $L^R$ have the same Jones polynomial.
\end{thm}
\begin{rem}
The Kauffman bracket definition of the Jones polynomial of a link has a factor determined by the link's writhe, and for a pair of links with the same bracket to have matching Jones polynomials, this factor must agree for each link. This is why the condition on orientation is included in Definition~\ref{defn:rotant}.
\end{rem}

Rolfsen and Jin showed that the bound $n=5$ cannot be improved upon, by exhibiting a pair of order 6 rotants with different Jones polynomials \cite{J-R}.
\\
The following question was posed in \cite{A-P-R}, and is still open:

\begin{ques}\label{ques:unknot rotant}
Does there exist an unknot with an $n$-rotant which is nontrivially knotted?
\end{ques}
Clearly, an affirmative answer to this question for $n\leq 5$ would imply one for Question~\ref{ques:jones}.
\\
\par

Khovanov homology, an invariant of oriented links, is a bigraded homology theory whose Euler characteristic is the (unnormalised) Jones polynomial \cite{Khovanov}. Examples of links with the same Jones polynomial but distinct Khovanov homologies, such as the knots $5_1$ and $10_{132}$ in Rolfsen's table \cite{Bar-Natan}, indicate that Khovanov homology is a strictly stronger invariant than the Jones polynomial. Theorem~\ref{thm:main} (which we will prove in Section~\ref{sec:proof}) makes use of rotants as a further demonstration of Khovanov homology's heightened sensitivity compared to the Jones polynomial.

\begin{thm}\label{thm:main}
There exist infinitely many rotant pairs distinguishable by Khovanov homology but not by the Jones polynomial.
\end{thm}

We can show that no use of Conway mutation can transform an unknot into a nontrivial knot by combining the facts that Khovanov homology detects the unknot \cite{K-M} and that it is insensitive to mutation (when calculated over $\mathbb{Z}/2\mathbb{Z}$) \cite{Bloom,Wehrli}. This means that Conway mutation cannot be used to provide an affirmative answer to Question~\ref{ques:jones}, as is well-known. On the other hand, Theorem~\ref{thm:main} tells us that Khovanov homology can distinguish certain rotants, and so Question~\ref{ques:unknot rotant} remains a reasonable approach to Question~\ref{ques:jones}

\begin{figure}[ htbp ]
\centering
\includegraphics [width =15 cm, height =12 cm]{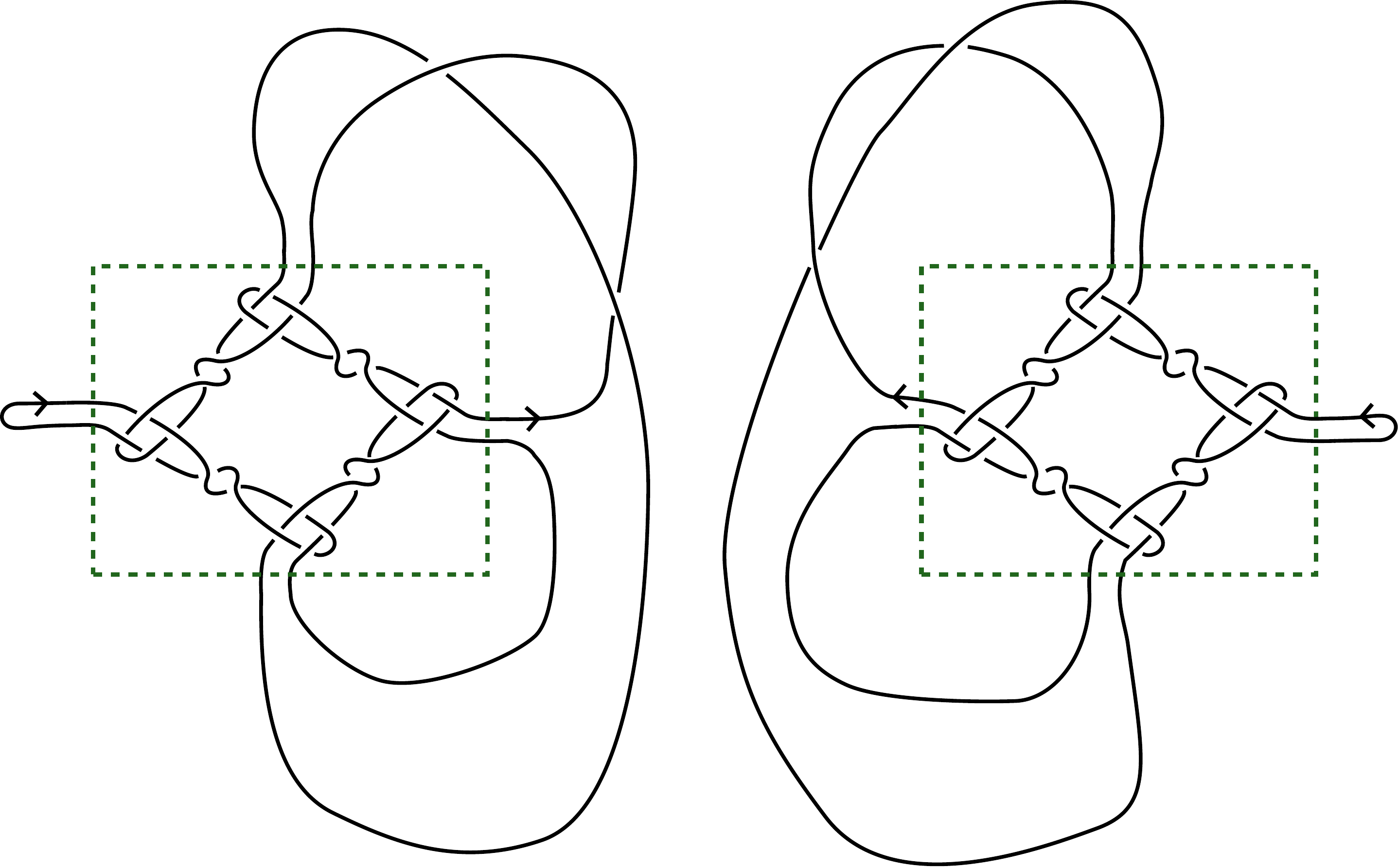}
\caption {A pair of order 4 rotants, adapted from \protect\cite[Example 1]{J-R}. Note that for clarity we have kept the rotor (inside the box) the same between the two diagrams and flipped the stator, a procedure equivalent to the one outlined in Definition~\ref{defn:rotant}. These rotants are distinguishable by Khovanov homology (see Section~\ref{sec:calcs}).\label{fig:first ex}}
\end{figure}

\section{Construction}\label{sec:proof}

We will now extend the example in Figure~\ref{fig:first ex} to an infinite family of pairs of order 4 rotants with different Khovanov homologies, by inserting additional twists between parallel strands in the diagrams. We specify pairs of rotants $(L(n),L^R(n)),\, n\in \mathbb{Z}$, by inserting $n$ right-handed half-twists at a  location in the stator, as shown in Figure~\ref{fig:param rotants}. With this notation, the pair of rotants presented in Figure~\ref{fig:first ex} is $(L(2),L^R(2))$.
\\
\\
We also introduce the following shorthand, used to describe the rotants in our infinite family:
\[
L_n:=L(-20-n)\hspace{30pt} L^R_n:=L^R(-20-n)
\]
where $n$ is a nonnegative integer.

\begin{figure}[ htbp ]
\centering
\includegraphics [width =15 cm, height =10 cm]{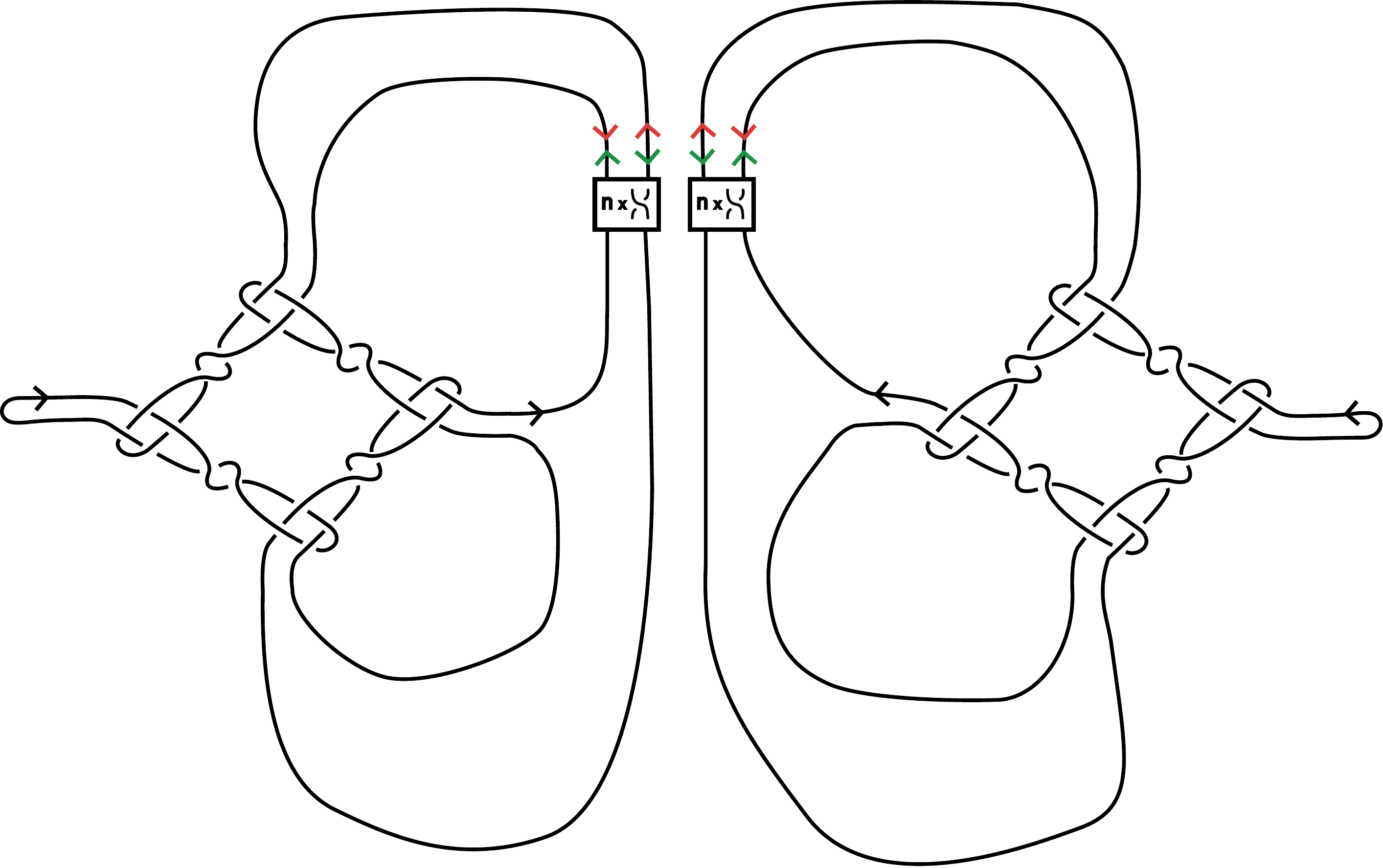}
\caption {The parametrised rotants $(L(n),L^R(n))$ we build our infinite family from. Note that the orientation above the parametrised twists depends on the parity of $n$. The green arrows show the orientation when $n$ is even, the red arrows when $n$ is odd.\label{fig:param rotants}}
\end{figure}

Our construction makes use of two key properties of Khovanov homology:

\begin{khprop}\label{khprop:kh orientation}
If $L$ and $L'$ are link diagrams that coincide, up to planar isotopy, as unoriented links, then $L$ and $L'$ will have identical Khovanov chain complexes up to an overall grading shift. This grading shift depends on the orientations of $L$ and $L'$ --- in particular it depends on the number of positive and negative crossings in each diagram. As a result, if these two quantities agree for both $L$ and $L'$, then $\text{Kh}(L)\cong\text{Kh}(L')$ \cite{Khovanov} (see also \cite{Turner}).
\end{khprop}

\begin{khprop}\label{khprop:l.e.s.}
There is a long exact sequence in Khovanov homology, described by Rasmussen \cite{Rasmussen} (see also \cite{Turner}). If we choose a crossing in $D$, the planar diagram of a link, then the sequence relates the Khovanov homologies of $D$ and the diagrams $D_0,\,D_1$ we get by giving this crossing a 0-resolution or a 1-resolution, respectively (see Figure~\ref{fig:crossing resolutions}). If the crossing resolved is oriented negatively, the long exact sequence, as presented by Turner, is:
\[
\dotsm\rightarrow\text{Kh}^t_{q+1}(D_1)\rightarrow\text{Kh}^t_q(D)\rightarrow\text{Kh}^{t-c}_{q-3c-1}(D_0)\rightarrow\text{Kh}^{t+1}_{q+1}(D_1)\rightarrow\dotsm
\]
where $c$ is defined:
\[
c= N_-(D_0)-N_-(D)
\]
The notation $\text{Kh}^t_q(L)$ refers to the vector space at quantum grading $q$ and homological grading $t$ in the Khovanov homology of the link $L$, and $N_-(L)$ is the number of negative crossings in (a diagram of) $L$.
\end{khprop}

\begin{figure}[ htpb ]
\centering
\includegraphics [width =7 cm, height =1 cm]{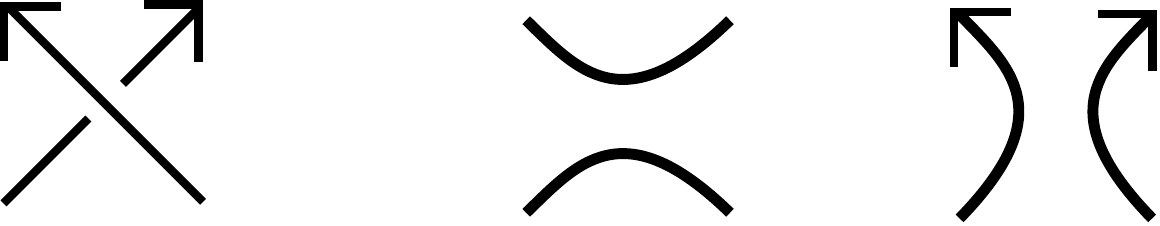}
\caption {The crossing configurations in the links we fit into the long exact sequence. Following the convention in \protect\cite{Turner}, from left to right we have: a crossing with negative orientation, its 0-resolution and its 1-resolution. Note that the 1-resolution inherits an orientation from the original crossing, while we can orient the affected strands of the 0-resolution as we please.\label{fig:crossing resolutions}}
\end{figure}

For any $n\in \mathbb{N}$, the crossings added by parametrisation to $L_n$ and $L^R_n$ are oriented negatively, and we can resolve the lowest of them in one of the two ways shown in Figure~\ref{fig:crossing resolutions}. The 0-resolution simply gives us $L_{n-1}$ or $L^R_{n-1}$, while the 1-resolution results in the 3-component rotants shown in Figure~\ref{fig:inf resolutions}, where we can always undo the remaining parametrised twists using type 1 Reidemeister moves. The orientation inherited on the newly formed component is dependent on the parity of $n$, as pictured. We call the rotants obtained from the oriented 1-resolution $(L^0_\infty, (L^0_\infty)^R)$ when $n$ is even, and $(L^1_\infty, (L^1_\infty)^R)$ when $n$ is odd.

\begin{figure}[ htbp ]
\centering
\includegraphics [width =15 cm, height =10 cm]{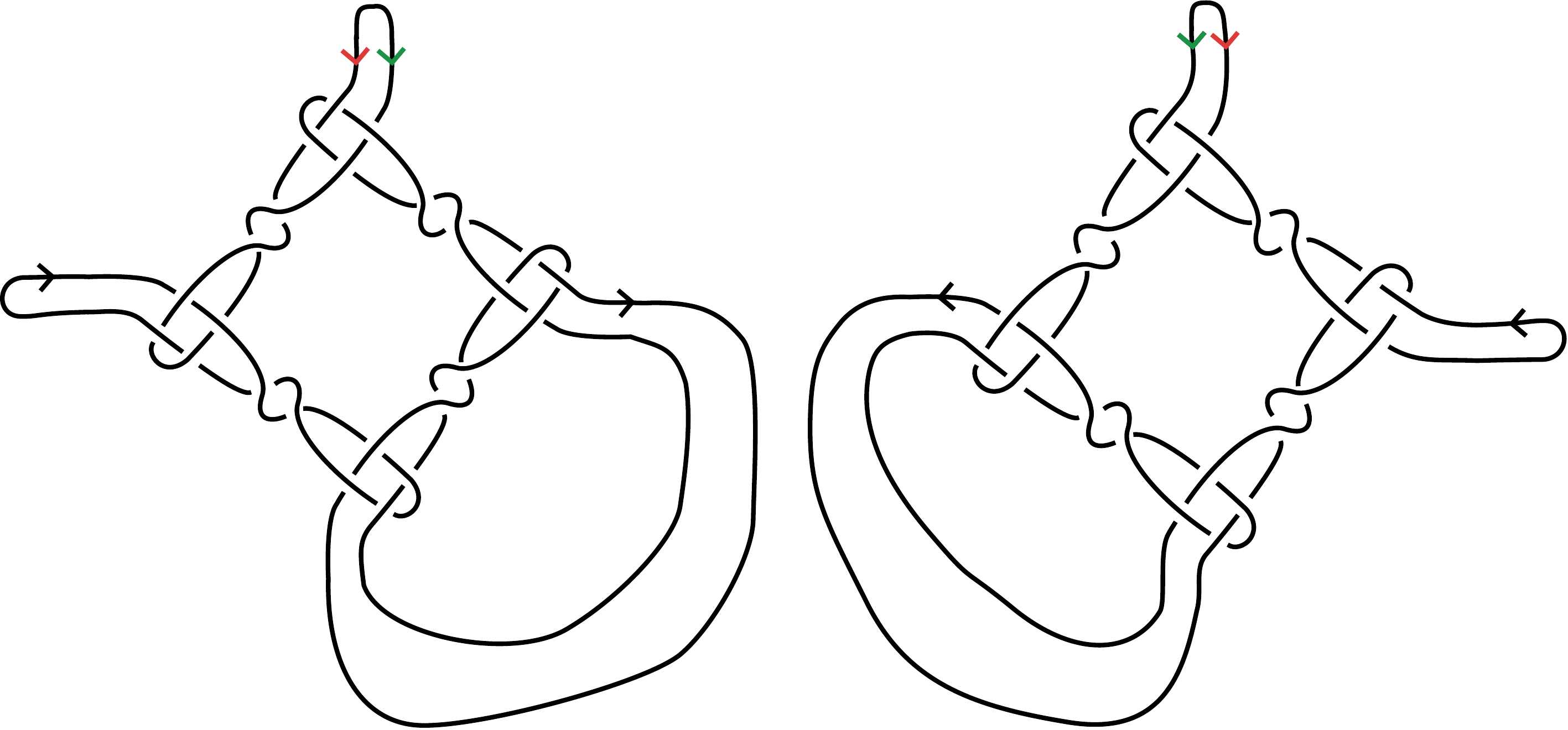}
\caption {The rotants we obtain from applying the 1-resolution to the lowest of the parametrised crossings in $(L_n,L^R_n)$. Taking the orientations given by the green arrows gives the pair $(L^0_\infty, (L^0_\infty)^R)$, while the red arrows give $(L^1_\infty, (L^1_\infty)^R)$, these corresponding to the cases when $n$ is even or odd, respectively. \label{fig:inf resolutions}}
\end{figure}

We observe three facts, essential to our construction:

\begin{fact}\label{fact:inf support}
$\text{Kh}^t_q(L^0_\infty)$ is trivial when $q<-53$ and $\text{Kh}^t_q(L^1_\infty)$ is trivial when $q<-29$.
\end{fact}

\begin{proof}
By direct calculation; see Section~\ref{sec:calcs}.
\end{proof}

\begin{fact}\label{fact:inf rotant homologies}
The pair of rotants $(L^0_\infty,(L^0_\infty)^R)$ have identical Khovanov homology, as do $(L^1_\infty,(L^1_\infty)^R)$.
\end{fact}
\begin{rem}
This means that the result of Fact~\ref{fact:inf support} also holds if we replace $L_\infty^0$ with $(L^0_\infty)^R$, and $L_\infty^1$ with $(L^1_\infty)^R$.
\end{rem}

\begin{proof}
As can be seen from Figure~\ref{fig:inf rotant equivalence}, the rotants $(L^0_\infty,(L^0_\infty)^R)$ are identical as unoriented links. As oriented links, they differ by the reversal of orientation on all components, meaning that the rotants both have the same numbers of positive and negative crossings. So, by Property~\ref{khprop:kh orientation}, these rotants must have identical Khovanov homologies. By a similar argument, the rotants $(L^1_\infty,(L^1_\infty)^R)$ have identical Khovanov homologies.
\end{proof}

\begin{figure}[ htbp ]
\centering
\includegraphics [width =15 cm, height =10 cm]{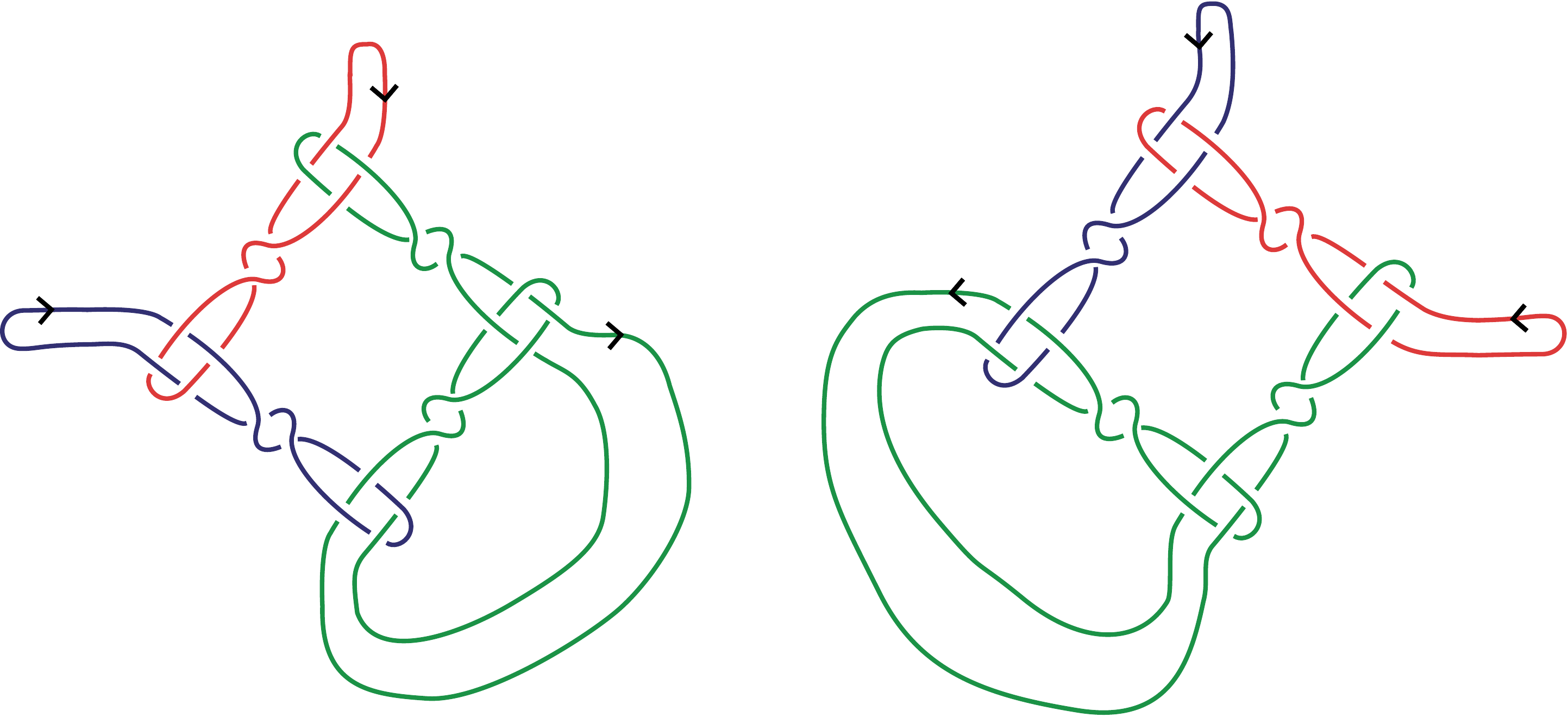}
\caption {The rotants $(L^0_\infty,(L^0_\infty)^R)$, coloured to show the interactions between components. An anticlockwise rotation of $(L^0_\infty)^R$ by $\pi /2$ gives us $L^0_\infty$ with the orientations reversed on each component. \label{fig:inf rotant equivalence}}
\end{figure}

\begin{fact}\label{fact:neg crossings}
In the long exact sequence of Property~\ref{khprop:l.e.s.}, if $D$ is $L_{n+1}$ for some nonnegative integer $n$, then $c$ is given by:
\[
c=
\begin{cases}
7\hspace{10pt}&\text{if }n\text{ is even}\\
-9&\text{if }n\text{ is odd}
\end{cases}
\]
\end{fact}

\begin{proof}
First we observe that, in the notation of Property~\ref{khprop:l.e.s.}, we have $D=L_{n+1}$ and $D_0=L_n$. So, by definition, $c=N_-(L_n)-N_-(L_{n+1})$. See the diagram of the parametrised link $L$ in Figure~\ref{fig:param rotants} for information on the crossing orientations of $L_n$. Since $n\ge 0$, the crossings added by parametrisation to $L_n=L(-20-n)$ are all oriented negatively, regardless of the parity of $n$. Further, we see that if $n$ is even then the 24 crossings which are not due to the added twists are all negative as well, giving $20+n+24=44+n$ negative crossings in total. To get $N_-(L_n)$ when $n$ is odd, we simply observe that changing the orientation induced by $n$ being even to the one for $n$ odd results in 8 of the crossings inside the rotor changing from negative to positive. Therefore we have:
\[
N_-(L_n)=
\begin{cases}
44+n\hspace{10pt}&\text{if }n\text{ is even}\\
36+n&\text{if }n\text{ is odd}
\end{cases}
\]
If $n$ is even, then $c=(44+n)-(36+n+1)=7$; similarly if $n$ is odd, $c=-9$.
\end{proof}

We can now use the family we have constructed to prove Theorem~\ref{thm:main}, as a consequence of:

\begin{prop}\label{prop:l.e.s. isom}
For any nonnegative integer $n$, and $t\in\mathbb{Z}$, we have
\[\text{Kh}_{Q(n+1)}^{t+c}(L_{n+1})\cong\text{Kh}_{Q(n)}^t(L_n)\]
and
\[\text{Kh}_{Q(n+1)}^{t+c}(L_{n+1}^R)\cong\text{Kh}_{Q(n)}^t(L_n^R)\]
where
\[Q(n)=-76+22\lceil n/2\rceil-26\lfloor n/2\rfloor\]
and $c$ is as described in Fact~\ref{fact:neg crossings}.
\end{prop}

\begin{proof}
The long exact sequence of Property~\ref{khprop:l.e.s.} gives us:
\[
\dotsm\rightarrow \text{Kh}^{t+c}_{q+3c+2}(L^{\overline{n+1}}_\infty) \rightarrow \text{Kh}^{t+c}_{q+3c+1}(L_{n+1}) \rightarrow \text{Kh}^t_q(L_n) \rightarrow \text{Kh}^{t+c+1}_{q+3c+2}(L^{\overline{n+1}}_\infty) \rightarrow\dotsm
\]
where $\overline{n+1}$ denotes the reduction of $n+1$ modulo 2.
\\
\\
Fact~\ref{fact:neg crossings} gives us the identity:
\[
Q(n)+3c+1=Q(n+1)
\]
for any $n$. Indeed, when $n$ is even:
\[
Q(n+1)=-76+22\lceil (n+1)/2\rceil -26\lfloor (n+1)/2\rfloor = -76+22(\lceil n/2\rceil +1)-26(\lfloor n/2\rceil)= Q(n)+22=Q(n)+3c+1
\]
We can similarly check the case when $n$ is odd. So with $q=Q(n)$ we get the long exact sequence:
\[
\dotsm\rightarrow \text{Kh}^{t+c}_{Q(n+1)+1}(L^{\overline{n+1}}_\infty) \rightarrow \text{Kh}^{t+c}_{Q(n+1)}(L_{n+1}) \rightarrow \text{Kh}^t_{Q(n)}(L_n) \rightarrow \text{Kh}^{t+c+1}_{Q(n+1)+1}(L^{\overline{n+1}}_\infty) \rightarrow\dotsm
\]
\\
Now, consider the expression $Q(n+1)+1$ --- when $n$ is even it evaluates to $-2n-53$, and when $n$ is odd it evaluates to $-2n-77$. So $Q(n+1)+1\le-53<-29$ if $n\ge0$ is even, and $Q(n+1)+1\le-80<-53$ if $n\ge1$ is odd. Therefore, by Fact~\ref{fact:inf support}, the vector spaces in the homology of $L^{\overline{n+1}}_\infty$ specified by this long exact sequence will be trivial for any $n\ge0$. As a result, the long exact sequence forces the desired isomorphism:
\[
\text{Kh}^{t+c}_{Q(n+1)}(L_{n+1})\cong\text{Kh}^t_{Q(n)}(L_n)
\]
\\
We can argue in the same way to get the result for $L^R_{n+1}$, making use of Fact~\ref{fact:inf rotant homologies}.
\end{proof}

\begin{proof}[Proof of Theorem~\ref{thm:main}]

Observe that $Q(0)=-76$ and, by direct calculation (see Section~\ref{sec:calcs}),
\[
\text{Kh}_{-76}^{-32}(L_0)\cong\mathbb{Q}^{75}\hspace{15pt}\text{while}\hspace{15pt}\text{Kh}_{-76}^{-32}(L_0^R)\cong\mathbb{Q}^{74}
\]
So Proposition~\ref{prop:l.e.s. isom} tells us that, for any nonnegative integer $n$, there is a vector space isomorphic to $\mathbb{Q}^{75}$ in the Khovanov homology of $L_n$, while, in the same homological and quantum gradings, the homology of $L_n^R$ has a vector space isomorphic to $\mathbb{Q}^{74}$. This means that the infinite family of rotant pairs $(L_n,L_n^R),\,n\ge0$, are all distinguishable by Khovanov homology. However, they are order $4$ rotants, and so by Theorem~\ref{thm:rotant jones}, they must all have matching Jones polynomials.
\end{proof}
\begin{rem}
We have shown a difference between the rotants in their Khovanov homology calculated over $\mathbb{Q}$, although a similar argument works for the homologies calculated with other coefficients (see Section~\ref{sec:calcs}).
\end{rem}

\section{Calculations}\label{sec:calcs}

Finally, we provide the data required to justify claims we have made about the Khovanov homologies of various rotants. All calculations were carried out using the Mathematica package \texttt{knottheory\`{}} \cite{katlas}. First, we find the vector spaces needed to prove the base result (for $(L_0,L_0^R)$) in the proof of Theorem~\ref{thm:main}. We give the vector spaces at quantum grading $-76$ in the Khovanov homologies of the rotants $(L_0,L^R_0)$, with coefficients in $\mathbb{Q}$, $\mathbb{Z}$ and $\mathbb{Z}/2\mathbb{Z}$:

\begin{enumerate}
\item
\[\text{Kh}^t_{-76}(L_0;\mathbb{Q})\cong
\begin{cases}
\mathbb{Q}^{75}\hspace{10pt}&t=-32\\
\mathbb{Q}^{127}&t=-31\\
\mathbb{Q}^{23}&t=-30\\
\{0\}&\text{otherwise}
\end{cases}
\hspace{30pt}
\text{Kh}^t_{-76}(L^R_0;\mathbb{Q})\cong
\begin{cases}
\mathbb{Q}^{74}\hspace{10pt}&t=-32\\
\mathbb{Q}^{124}&t=-31\\
\mathbb{Q}^{21}&t=-30\\
\{0\}&\text{otherwise}
\end{cases}
\]

\item
\[\text{Kh}^t_{-76}(L_0;\mathbb{Z})\cong
\begin{cases}
\mathbb{Z}^{75}\hspace{10pt}&t=-32\\
\mathbb{Z}^{127}\oplus(\mathbb{Z}/2\mathbb{Z})^{99}&t=-31\\
\mathbb{Z}^{23}\oplus(\mathbb{Z}/2\mathbb{Z})^{11}&t=-30\\
\{0\}&\text{otherwise}
\end{cases}
\hspace{30pt}
\text{Kh}^t_{-76}(L^R_0;\mathbb{Z})\cong
\begin{cases}
\mathbb{Z}^{74}\hspace{10pt}&t=-32\\
\mathbb{Z}^{124}\oplus(\mathbb{Z}/2\mathbb{Z})^{99}&t=-31\\
\mathbb{Z}^{21}\oplus(\mathbb{Z}/2\mathbb{Z})^{11}&t=-30\\
\{0\}&\text{otherwise}
\end{cases}
\]

\item Let $\mathbb{F}\cong\mathbb{Z}/2\mathbb{Z}$ be the field with two elements.
\[\text{Kh}^t_{-76}(L_0;\mathbb{F})\cong
\begin{cases}
(\mathbb{F})^{174}\hspace{10pt}&t=-32\\
(\mathbb{F})^{237}&t=-31\\
(\mathbb{F})^{34}&t=-30\\
\{0\}&\text{otherwise}
\end{cases}
\hspace{30pt}
\text{Kh}^t_{-76}(L^R_0;\mathbb{F})\cong
\begin{cases}
(\mathbb{F})^{173}\hspace{10pt}&t=-32\\
(\mathbb{F})^{234}&t=-31\\
(\mathbb{F})^{32}&t=-30\\
\{0\}&\text{otherwise}
\end{cases}
\]
\end{enumerate}

We note that $\dim(\text{Kh}^{-32}_{-76}(L_0))\neq \dim(\text{Kh}^{-32}_{-76}(L^R_0))$ for each set of coefficients, meaning that we can use any of these to make an argument analogous to the proof of Theorem~\ref{thm:main}.
\\
\\
The following pages give data on the complete Khovanov homologies of the rotant pair $(L(2),L^R(2))$ from Figure~\ref{fig:first ex} and the links $L^0_\infty ,\, L^1_\infty$, calculated over $\mathbb{Q}$. We note that the quantum grading supports for nontrivial vector spaces in the homologies of  $L^0_\infty$ and $L^1_\infty$ are the same when calculated with any coefficients, so that Fact~\ref{fact:inf support} still applies and can be used to prove Theorem~\ref{thm:main} for coefficients in, say, $\mathbb{Z}$ or $\mathbb{F}$ rather than $\mathbb{Q}$. The number in cell $(q,t)$ gives the dimension of the vector space at quantum grading $q$ and homological grading $t$ --- where there is no number, the vector space is trivial. Notice that the sum of the dimensions of all the vector spaces in $\text{Kh}(L(2))$ is greater than that of $\text{Kh}(L^R(2))$, and that for any $t,\, q\in\mathbb{Z}$, we have $\dim(\text{Kh}^t_q(L(2)))\ge\dim(\text{Kh}^t_q(L^R(2)))$. This appears to be the case in general for rotants in this family (i.e. for any value of $n$, not just $2$). Also, the homologies of $L^0_\infty$ and $L^1_\infty$ are identical up to a grading shift, since the links differ only in orientation. 
\newpage
\[\underline{\text{Kh}(L(2);\mathbb{Q}):}\]

\begin{footnotesize}
\[\arraycolsep=2pt
\begin{array}{cccccccccccccccccccccc}
 \text{q$\backslash $t}\:\:\:\vline & -20 & -19 & -18 & -17 & -16 & -15 & -14 & -13 & -12 & -11 & -10 & -9 & -8 & -7 & -6 & -5 & -4 & -3 & -2 & -1 & 0 \\
\hline
 -52\;\vline & 1 &  &  &  &  &  &  &  &  &  &  &  &  &  &  &  &  &  &  &  &  \\
 -50 \;\vline&  & 4 &  &  &  &  &  &  &  &  &  &  &  &  &  &  &  &  &  &  &  \\
 -48 \;\vline&  & 1 & 10 &  &  &  &  &  &  &  &  &  &  &  &  &  &  &  &  &  &  \\
 -46 \;\vline&  &  & 4 & 20 &  &  &  &  &  &  &  &  &  &  &  &  &  &  &  &  &  \\
 -44 \;\vline&  &  &  & 10 & 31 & 1 &  &  &  &  &  &  &  &  &  &  &  &  &  &  &  \\
 -42 \;\vline&  &  &  &  & 20 & 42 & 4 &  &  &  &  &  &  &  &  &  &  &  &  &  &  \\
 -40 \;\vline&  &  &  &  &  & 31 & 51 & 10 &  &  &  &  &  &  &  &  &  &  &  &  &  \\
 -38 \;\vline&  &  &  &  &  &  & 42 & 56 & 19 &  &  &  &  &  &  &  &  &  &  &  &  \\
 -36 \;\vline&  &  &  &  &  &  &  & 50 & 57 & 28 &  &  &  &  &  &  &  &  &  &  &  \\
 -34 \;\vline&  &  &  &  &  &  &  &  & 52 & 59 & 38 &  &  &  &  &  &  &  &  &  &  \\
 -32 \;\vline&  &  &  &  &  &  &  &  &  & 47 & 55 & 42 &  &  &  &  &  &  &  &  &  \\
 -30 \;\vline&  &  &  &  &  &  &  &  &  &  & 40 & 55 & 43 &  &  &  &  &  &  &  &  \\
 -28 \;\vline&  &  &  &  &  &  &  &  &  &  &  & 27 & 50 & 39 &  &  &  &  &  &  &  \\
 -26 \;\vline&  &  &  &  &  &  &  &  &  &  &  &  & 18 & 47 & 29 &  &  &  &  &  &  \\
 -24 \;\vline&  &  &  &  &  &  &  &  &  &  &  &  &  & 7 & 40 & 22 &  &  &  &  &  \\
 -22 \;\vline&  &  &  &  &  &  &  &  &  &  &  &  &  &  & 4 & 30 & 11 &  &  &  &  \\
 -20 \;\vline&  &  &  &  &  &  &  &  &  &  &  &  &  &  &  & 1 & 22 & 6 &  &  &  \\
 -18 \;\vline&  &  &  &  &  &  &  &  &  &  &  &  &  &  &  &  & 1 & 11 & 2 &  &  \\
 -16 \;\vline&  &  &  &  &  &  &  &  &  &  &  &  &  &  &  &  &  &  & 6 &  &  \\
 -14 \;\vline&  &  &  &  &  &  &  &  &  &  &  &  &  &  &  &  &  &  &  & 2 & 1 \\
 -12 \;\vline&  &  &  &  &  &  &  &  &  &  &  &  &  &  &  &  &  &  &  &  & 1 \\
\end{array}\]
\end{footnotesize}

\[\underline{\text{Kh}(L^R(2);\mathbb{Q}):}\]

\begin{footnotesize}
\[\arraycolsep=2pt
\begin{array}{cccccccccccccccccccccc}
 \text{q$\backslash $t}\:\:\vline & -20 & -19 & -18 & -17 & -16 & -15 & -14 & -13 & -12 & -11 & -10 & -9 & -8 & -7 & -6 & -5 & -4 & -3 & -2 & -1 & 0 \\
\hline
 -52\vline & 1 &  &  &  &  &  &  &  &  &  &  &  &  &  &  &  &  &  &  &  &  \\
 -50\vline &  & 4 &  &  &  &  &  &  &  &  &  &  &  &  &  &  &  &  &  &  &  \\
 -48\vline &  & 1 & 10 &  &  &  &  &  &  &  &  &  &  &  &  &  &  &  &  &  &  \\
 -46\vline &  &  & 4 & 20 &  &  &  &  &  &  &  &  &  &  &  &  &  &  &  &  &  \\
 -44\vline &  &  &  & 10 & 31 & 1 &  &  &  &  &  &  &  &  &  &  &  &  &  &  &  \\
 -42\vline &  &  &  &  & 20 & 41 & 3 &  &  &  &  &  &  &  &  &  &  &  &  &  &  \\
 -40\vline &  &  &  &  &  & 31 & 50 & 9 &  &  &  &  &  &  &  &  &  &  &  &  &  \\
 -38\vline &  &  &  &  &  &  & 41 & 53 & 17 &  &  &  &  &  &  &  &  &  &  &  &  \\
 -36\vline &  &  &  &  &  &  &  & 49 & 52 & 24 &  &  &  &  &  &  &  &  &  &  &  \\
 -34\vline &  &  &  &  &  &  &  &  & 50 & 53 & 34 &  &  &  &  &  &  &  &  &  &  \\
 -32\vline &  &  &  &  &  &  &  &  &  & 43 & 47 & 38 &  &  &  &  &  &  &  &  &  \\
 -30\vline &  &  &  &  &  &  &  &  &  &  & 36 & 47 & 39 &  &  &  &  &  &  &  &  \\
 -28\vline &  &  &  &  &  &  &  &  &  &  &  & 23 & 44 & 37 &  &  &  &  &  &  &  \\
 -26\vline &  &  &  &  &  &  &  &  &  &  &  &  & 13 & 39 & 26 &  &  &  &  &  &  \\
 -24\vline &  &  &  &  &  &  &  &  &  &  &  &  &  & 5 & 37 & 21 &  &  &  &  &  \\
 -22\vline &  &  &  &  &  &  &  &  &  &  &  &  &  &  & 1 & 26 & 10 &  &  &  &  \\
 -20\vline &  &  &  &  &  &  &  &  &  &  &  &  &  &  &  &  & 21 & 6 &  &  &  \\
 -18\vline &  &  &  &  &  &  &  &  &  &  &  &  &  &  &  &  &  & 10 & 2 &  &  \\
 -16\vline &  &  &  &  &  &  &  &  &  &  &  &  &  &  &  &  &  &  & 6 &  &  \\
 -14\vline &  &  &  &  &  &  &  &  &  &  &  &  &  &  &  &  &  &  &  & 2 & 1 \\
 -12\vline &  &  &  &  &  &  &  &  &  &  &  &  &  &  &  &  &  &  &  &  & 1 \\
\end{array}\]
\end{footnotesize}

For interest, we also record the common reduced Jones polynomial of these rotants:

\begin{align*}
V_{L(2)}(q)=V_{L^R(2)}(q)=q^{-51/2}(&-1+5q-14q^2+30q^3-50q^4+68q^5-78q^6+73q^7-52q^8+21q^9+13q^{10}\\
& -41q^{11}+57q^{12}-57q^{13}+46q^{14}-31q^{15}+16q^{16}-8q^{17}+2q^{18}-q^{19})
\end{align*}

\newpage

\[\underline{\text{Kh}(L^0_\infty;\mathbb{Q}):}\]

\begin{footnotesize}
\[\arraycolsep=2pt
\begin{array}{cccccccccccccccccccccc}
 \text{q$\backslash $t}\:\:\vline & -20 & -19 & -18 & -17 & -16 & -15 & -14 & -13 & -12 & -11 & -10 & -9 & -8 & -7 & -6 & -5 & -4 & -3 & -2 & -1 & 0 \\
\hline
 -53\vline & 1 &  &  &  &  &  &  &  &  &  &  &  &  &  &  &  &  &  &  &  &  \\
 -51\vline &  & 3 &  &  &  &  &  &  &  &  &  &  &  &  &  &  &  &  &  &  &  \\
 -49\vline &  & 1 & 6 &  &  &  &  &  &  &  &  &  &  &  &  &  &  &  &  &  &  \\
 -47\vline &  &  & 3 & 11 &  &  &  &  &  &  &  &  &  &  &  &  &  &  &  &  &  \\
 -45\vline &  &  &  & 6 & 15 & 1 &  &  &  &  &  &  &  &  &  &  &  &  &  &  &  \\
 -43\vline &  &  &  &  & 11 & 19 & 2 &  &  &  &  &  &  &  &  &  &  &  &  &  &  \\
 -41\vline &  &  &  &  &  & 15 & 23 & 5 &  &  &  &  &  &  &  &  &  &  &  &  &  \\
 -39\vline &  &  &  &  &  &  & 19 & 24 & 9 &  &  &  &  &  &  &  &  &  &  &  &  \\
 -37\vline &  &  &  &  &  &  &  & 22 & 22 & 10 &  &  &  &  &  &  &  &  &  &  &  \\
 -35\vline &  &  &  &  &  &  &  &  & 22 & 25 & 15 &  &  &  &  &  &  &  &  &  &  \\
 -33\vline &  &  &  &  &  &  &  &  &  & 17 & 19 & 15 &  &  &  &  &  &  &  &  &  \\
 -31\vline &  &  &  &  &  &  &  &  &  &  & 16 & 21 & 15 &  &  &  &  &  &  &  &  \\
 -29\vline &  &  &  &  &  &  &  &  &  &  &  & 9 & 19 & 15 &  &  &  &  &  &  &  \\
 -27\vline &  &  &  &  &  &  &  &  &  &  &  &  & 8 & 17 & 9 &  &  &  &  &  &  \\
 -25\vline &  &  &  &  &  &  &  &  &  &  &  &  &  & 1 & 15 & 9 &  &  &  &  &  \\
 -23\vline &  &  &  &  &  &  &  &  &  &  &  &  &  &  & 3 & 10 & 4 &  &  &  &  \\
 -21\vline &  &  &  &  &  &  &  &  &  &  &  &  &  &  &  &  & 9 & 3 &  &  &  \\
 -19\vline &  &  &  &  &  &  &  &  &  &  &  &  &  &  &  &  & 1 & 4 & 1 &  &  \\
 -17\vline &  &  &  &  &  &  &  &  &  &  &  &  &  &  &  &  &  &  & 3 &  &  \\
 -15\vline &  &  &  &  &  &  &  &  &  &  &  &  &  &  &  &  &  &  &  & 1 & 1 \\
 -13\vline &  &  &  &  &  &  &  &  &  &  &  &  &  &  &  &  &  &  &  &  & 1 \\
\end{array}\]
\end{footnotesize}

\[\underline{\text{Kh}(L^1_\infty;\mathbb{Q}):}\]

\begin{footnotesize}
\[\arraycolsep=4pt
\begin{array}{cccccccccccccccccccccc}
 \text{q$\backslash $t}\:\:\vline & -12 & -11 & -10 & -9 & -8 & -7 & -6 & -5 & -4 & -3 & -2 & -1 & 0 & 1 & 2 & 3 & 4 & 5 & 6 & 7 & 8 \\
\hline
 -29\vline & 1 &  &  &  &  &  &  &  &  &  &  &  &  &  &  &  &  &  &  &  &  \\
 -27\vline &  & 3 &  &  &  &  &  &  &  &  &  &  &  &  &  &  &  &  &  &  &  \\
 -25\vline &  & 1 & 6 &  &  &  &  &  &  &  &  &  &  &  &  &  &  &  &  &  &  \\
 -23\vline &  &  & 3 & 11 &  &  &  &  &  &  &  &  &  &  &  &  &  &  &  &  &  \\
 -21\vline &  &  &  & 6 & 15 & 1 &  &  &  &  &  &  &  &  &  &  &  &  &  &  &  \\
 -19\vline &  &  &  &  & 11 & 19 & 2 &  &  &  &  &  &  &  &  &  &  &  &  &  &  \\
 -17\vline &  &  &  &  &  & 15 & 23 & 5 &  &  &  &  &  &  &  &  &  &  &  &  &  \\
 -15\vline &  &  &  &  &  &  & 19 & 24 & 9 &  &  &  &  &  &  &  &  &  &  &  &  \\
 -13\vline &  &  &  &  &  &  &  & 22 & 22 & 10 &  &  &  &  &  &  &  &  &  &  &  \\
 -11\vline &  &  &  &  &  &  &  &  & 22 & 25 & 15 &  &  &  &  &  &  &  &  &  &  \\
 -9 \hspace{1pt}\:\,\vline&  &  &  &  &  &  &  &  &  & 17 & 19 & 15 &  &  &  &  &  &  &  &  &  \\
 -7 \hspace{1pt}\:\,\vline&  &  &  &  &  &  &  &  &  &  & 16 & 21 & 15 &  &  &  &  &  &  &  &  \\
 -5 \hspace{1pt}\:\,\vline&  &  &  &  &  &  &  &  &  &  &  & 9 & 19 & 15 &  &  &  &  &  &  &  \\
 -3 \hspace{1pt}\:\,\vline&  &  &  &  &  &  &  &  &  &  &  &  & 8 & 17 & 9 &  &  &  &  &  &  \\
 -1 \hspace{1pt}\:\,\vline&  &  &  &  &  &  &  &  &  &  &  &  &  & 1 & 15 & 9 &  &  &  &  &  \\
 \hspace{6pt}1\hspace{1pt}\:\,\vline&  &  &  &  &  &  &  &  &  &  &  &  &  &  & 3 & 10 & 4 &  &  &  &  \\
 \hspace{6pt}3\hspace{1pt}\:\,\vline&  &  &  &  &  &  &  &  &  &  &  &  &  &  &  &  & 9 & 3 &  &  &  \\
 \hspace{6pt}5\hspace{1pt}\:\,\vline&  &  &  &  &  &  &  &  &  &  &  &  &  &  &  &  & 1 & 4 & 1 &  &  \\
 \hspace{6pt}7\hspace{1pt}\:\,\vline&  &  &  &  &  &  &  &  &  &  &  &  &  &  &  &  &  &  & 3 &  &  \\
 \hspace{6pt}9\hspace{1pt}\:\,\vline&  &  &  &  &  &  &  &  &  &  &  &  &  &  &  &  &  &  &  & 1 & 1 \\
 \hspace{2pt}11\hspace{1pt}\:\,\vline &  &  &  &  &  &  &  &  &  &  &  &  &  &  &  &  &  &  &  &  & 1 \\
\end{array}\]
\end{footnotesize}
\newpage

\section*{Acknowledgements}\label{sec:ack}
I would like to thank the University of Glasgow and the Carnegie Trust for giving me the opportunity to undertake this research and providing financial support. I am also grateful for the guidance of my supervisor, Liam Watson, as well as for the knowledge and enthusiasm he shared with me throughout this project. Finally I would like to thank Duncan McCoy for valuable feedback on an earlier draft.

\bibliographystyle{plain} \bibliography{REFERENCES}

\end{document}